%% file: zhikov_CE_arx.tex
\documentclass[12pt]{article}

\usepackage{amsmath,amssymb,amsthm,mathtools,enumitem}
\usepackage{geometry}
\usepackage{srcltx,authblk}
\input{math-e}

\newtheorem{Theorem}{Theorem}[section]

\newtheorem{Lemma}{Lemma}[section]
\newtheorem{Remark}{Remark}[section]
\newtheorem{Definition}{Definition}[section]

\newtheorem{Assumption}{Assumption}[section]

\begin{document}

\title{nmma584}

\def\note#1{\marginpar{\small #1}}

\author{team}



\newcommand{\dashint}{-\!\!\!\!\!\!\int}
\newcommand{\qs}{q^\star}

\newcommand{\vv}{\mathcal V}
\newcommand{\eS}{\mathbb S}
\newcommand{\loc}{{loc}}
\newcommand{\eR}{\mathbf{R}}
\newcommand{\Rd}{{\eR}^d}
\newcommand{\Rdd}{{\eR}^{d\times d}}
\newcommand{\Rdsym}{{\eR}^{d\times d}_{sym}}
\newcommand{\eN}{\mathbf{N}}
\newcommand{\eZ}{\mathbf{Z}}
\newcommand{\nm}[1]{\left\|#1\right\|}
\newcommand{\dual}[1]{\left\langle#1\right\rangle}
\newcommand{\spt}{\operatorname{spt}}
\newcommand{\csubset}{\subset\subset}

\newcommand{\dd}{\mbox{d}}
\newcommand{\ltws}{local in time weak solution}
\setlength{\jot}{2.5ex}
\newcommand{\isigma}{s}
\renewcommand{\ss}{\tau}
\newcommand{\sss}{s}

\newcommand{\Du}{\bD v}
\newcommand{\du}{\nabla v}
\newcommand{\xxl}{\mathcal{X}_{\ell}}
\newcommand{\wwl}{\mathcal{W}_{\ell}}
\newcommand{\bbl}{\mathcal{B}_{\ell}}
\newcommand{\atl}{\mathcal{A}_{\ell}}
\newcommand{\etl}{\mathcal{E}_{\ell}}
\newcommand{\etr}{\mathcal{E}}
\newcommand{\ddt}{\frac{d}{dt}}
\newcommand{\dt}{\partial_t}
\newcommand{\norm}[3]{\|{#1}\|_{#2}^{#3}}
\newcommand{\dalka}[3]{\operatorname{dist}_{#3}(#1,#2)}
\newcommand{\eend}[1]{\hfill \ensuremath{\Box}}
\newcommand{\beggin}[1]{Proof. }
\def\dx{\; \mathrm{d}x}
\def\dy{\; \mathrm{d}y}
\def\diff{\mathsf{d}}
\def\diver{\mathop{\mathrm{div}}\nolimits}
\def\diam{\mathrm{diam}}

\def\beps{\mathbf{B}^\varepsilon}
\def\ueps{\mathbf{u}^\varepsilon}
\def\uepskl{\mathbf{u}^{k,l}}
\newcommand{\uk}{\boldu^k}
\newcommand{\bk}{\mathbf{B}^k}
\def\uepspr{\mathbf{u}^{\varepsilon'}}

\def\boldxi{\boldsymbol{\xi}}
\def\Seps{\mathbf{S}^\varepsilon}
\def\Sepsk{\mathbf{S}^k}
\def\GEps{g^\varepsilon}
\def\GEpsPr{g^{\varepsilon'}}
\def\Sepspr{\boldS^{\varepsilon'}}
\def\piepspro{\pi^{\varepsilon',1}}
\def\bpsi{\boldsymbol{\psi}}
\def\boldu{\mathbf{u}}
\def\boldg{\mathbf{g}}
\def\boldh{\mathbf{h}}
\def\boldv{\mathbf{v}}
\def\boldw{\mathbf{w}}
\def\boldA{\mathbf{A}}
\def\boldD{\mathbf{D}}
\def\boldF{\mathbf{F}}
\def\boldH{\mathbf{H}}
\def\boldI{\mathbf{I}}
\def\boldO{\mathbf{O}}
\def\boldP{\mathbf{P}}
\def\boldS{\mathbf{S}}
\def\boldV{\mathbf{V}}
\def\boldW{\mathbf{W}}
\def\boldeta{\boldsymbol{\eta}}
\def\WTSCon{\xrightharpoonup{2-s}}
\def\STSCon{\xrightarrow{2-s}}
\def\pieps{\pi^\varepsilon}
\newcommand{\pik}{\pi^k}
\def\piepspr{\pi^{\varepsilon'}}
\def\piepso{\pi^{\varepsilon,1}}
\def\piepsok{\pi^{k,1}}
\newcommand{\tcj}{\tilde c_1}


\title{Homogenization of an incompressible stationary flow of an~electrorheological fluid}

\author[a]{Miroslav Bul\'\i\v cek}
\author[b]{Martin Kalousek}
\author[b]{Petr Kaplick\'y}
\affil[a]{\footnotesize Mathematical Institute, Charles University in Prague, Sokolovsk\'{a}~83, 18675 Praha~8 - Karl\'{\i}n, Czech Republic}
\affil[b]{\footnotesize Department of Mathematical Analysis, Charles University in Prague, Sokolovsk\'{a}~83, 18675 Praha 8 - Karl\'{\i}n, Czech Republic}

\date{}
\maketitle

\begin{abstract}
We combine two scale convergence, theory of monotone operators and results on approximation of Sobolev functions by Lipschitz functions to prove a homogenization process for an incompressible flow of a generalized Newtonian fluid. We avoid the necessity of testing the weak formulation of the initial and homogenized systems by corresponding weak solutions, which allows mild assumptions on lower bound for a growth of the elliptic term.  We show that the stress tensor for homogenized problem depends on the symmetric part of the velocity gradient involving the limit of a sequence selected from a family of solutions of initial problems.
\end{abstract}

\section{Introduction}
Electrorheological fluids are special liquids characterized by their ability to change significantly the mechanical properties when an electric field is applied. This behavior has been extensively investigated for the development of smart fluids, which are currently exploited in technological applications, e.g. brakes, clutches or shock absorbers. Results of the ongoing research indicate their possible applications also in electronics. One approach for modeling of the flow of electrorheological fluids is the utilization of a system of partial differential equations derived by Rajagopal and R\r{u}\v{z}i\v{c}ka, for details see \cite{RuzK}. This system in the case of an isothermal, homogeneous (with density equal to one), incompressible electrorheological fluid reads
\begin{equation}\label{ERNS}
	\begin{split}
		\partial_t\boldu -\div \boldS +\div(\boldu\otimes\boldu)+\nabla\pi=\mathbf{f},\quad
		\div \boldu=0,
	\end{split}
\end{equation}
in a domain $\Omega\subset\eR^d$, $d=2,3,\dots$. The symbol $\boldu$ denotes the velocity, $\boldS$ the extra stress tensor, $\div(\boldu\otimes\boldu)$ is the convective term with $\boldu\otimes\boldu$ denoting the tensor product of the vector $\boldu$ with itself defined as $(u_iu_j)_{i,j=1,\ldots,n}$, $\pi$ is the pressure and $\mathbf{f}$ the external body force. The stress tensor $\boldS$ is assumed to depend on the symmetric part $\boldD\boldu$ of the velocity gradient $\nabla\boldu$. The presence of an electric field is captured by the supposed dependence of $\boldS$ on the spatial variable in such a way that the growth of $\boldS$ corresponds to $|\boldD\boldu|^{p(\cdot)-1}$ for some variable exponent $p$. 

For this setting assuming additionally a periodic variable exponent with a small period $\varepsilon$, it was shown by Zhikov in \cite{Zhikov} that as $\varepsilon\rightarrow 0$ a subsequence of solutions of initial problems converges to a solution of the homogenized problem having the extra stress tensor independent of the spatial variable. Zhikov's approach is based on the fact that the regularity of solutions of the initial as well as homogenized problem allows to use these solutions as a test function. In fact, this sufficient regularity is ensured by the value of the lower bound for the variable exponent $p\geq p_0:=\max((d+\sqrt{3d^2+4d})/(d+2),3d/(d+2))$. 

In the seminal article \cite{FMS} a method of Lipschitz approximation of Sobolev functions is developed that allows to decrease the lower bound for $p$. In the article  \cite{FMS} the method is applied to the problem of existence of a weak solution to the  stationary generalized Navier-Stokes model. The stationary problem with elliptic operator with Orlicz growth is studied in \cite{DMS}. It took lot of work till the approach was modified in such a way that it is applicable to evolutionary problems. See \cite{DRW}, where the existence of a weak solution to the evolutionary generalized Navier Stokes problem is studied. The method is used to an evolutionary problem in Orlicz setting in the article \cite{BGMS}. The existence of a solutions to the problem \eqref{ERNS} can be shown if $p>2d/(d+2)$. It is natural to ask: ``Can one proceed with the homogenization process also if the lower bound for $p$ is between $p_0$ and $2d/(d+2)$?'' This paper should be regarded as the first step on the way for the answer to this question. To concentrate on the interplay between method of Lipschitz approximations and two scale convergence we start with the stationary problem first. 

Let us introduce the problem, which we deal with. The domain $\Omega\subset\Rd$, $d=2,3,\dots$ is supposed to be bounded and Lipschitz,  $Y=(0,1]^d$. For $\varepsilon\in(0,1)$ we consider the following stationary version of the problem~\eqref{ERNS}
\begin{equation}\label{ProbEps}
	\begin{gathered}
	-\div \left(\Seps(x,\boldD \ueps)-\ueps\otimes \ueps\right)+\nabla\pi^\varepsilon=-\div \boldF,\quad\div \ueps=0\quad\text{ in }\Omega,\\
	\ueps=0\quad\text{ on }\partial\Omega,\quad
	\int_\Omega \pi^\varepsilon=0.
      \end{gathered}
    \end{equation}
The function $\Seps$ is for any $x\in\eR^d$ and $\boldD\in\eR^{d\times d}_{sym}$ given by $\Seps(x,\boldD)=\boldS(x/\varepsilon,\boldD)$, where the tensor $\boldS:\Rd\times \Rdsym\rightarrow\Rdsym$ satisfies
\begin{Assumption}\label{Ass:S}
\leavevmode
\begin{enumerate}[label=(S\arabic*)]
	\item $\boldS$ is $Y-$periodic in the first variable, i.e., periodic in each argument $y_i, i=1,\ldots,d$ with the period 1, and continuous in the first variable,
	\item $\boldS$ is a Carath\'eodory function, i.e., $\boldS(\cdot,\boldxi)$ is measurable for all $\boldxi\in\Rdsym$, $\boldS(y,\cdot)$ is continuous for almost all $y\in\Rd$,
	\item \label{SStrMon} for $\boldxi_1,\boldxi_2\in\eR^{d\times d}_{sym}$, $\boldxi_1\neq\boldxi_2$ and a.a. $y\in\Rd$, $\left(\boldS(y,\boldxi_1)-\boldS(y,\boldxi_2)\right):\left(\boldxi_1-\boldxi_2\right)>0$,
	\item	\label{SCoercGrow} there are $p>1$, $p'=p/(p-1)$, $c_1, \tcj ,c_2>0$ that for all $y\in\eR^d$, $\boldxi\in\eR^{d\times s}_{sym}$
$$\boldS(y,\boldxi):\boldxi\geq c_1|\boldxi|^{p}-\tcj,\ |\boldS(y,\boldxi)|^{p'}\leq c_2(|\boldxi|^{p}+1).
$$

\end{enumerate}
\end{Assumption}

Typical example of a stress tensor satisfying Assumption~\ref{Ass:S} is 
$$
\boldS(y,\boldxi)=\alpha(y)(\delta+|\boldxi|^2)^{\frac{p-2}2}\boldxi+\beta(y)(\delta+|\boldxi|^2)^{\frac{\gamma(y)-2}2}\boldxi\quad\mbox{for $y\in\eR^d$, $\boldxi\in\eR^{d\times s}_{sym}$,}
$$
with $p>1$, $\delta\geq0$, $\alpha,\beta,\gamma:\Rd\to\eR$ periodic with respect to $Y$ and continuous. Moreover we assume that there exist $\alpha_0,\alpha_1,\beta_0,\beta_1, \gamma_0,\gamma_1\in\eR$ such that for all $y\in\Rd$, $0<\alpha_0\leq\alpha(y)\leq\alpha_1$, $0\leq\beta(y)\leq\beta_0$, $1\leq\gamma_0\leq\gamma(y)\leq\gamma_1<p$. It is motivated by the model derived  for a flow of the  electrorheological fluids by Rajagopal and R\r{u}\v{z}i\v{c}ka, for details see \cite{RuzK}, in particular \cite[Lemma~4.46]{RuzK}.

Our effort will be spent on determining the homogenized stress tensor. The situation is similar to the limit passage in the stress tensor in the proof of the existence of weak solutions of generalized Navier-Stokes equations. However, one cannot straightforwardly adopt the methods, which are successfully applied for existence proofs, because of oscillations which occur in the spatial variable of the stress tensor. We prove

\begin{Theorem}\label{Thm:Homog}
Let $\Omega\subset\Rd$ be a bounded Lipschitz domain, $p>2d/(d+2)$, $\boldS$ satisfy Assumption~\ref{Ass:S} and $\boldF\in L^{p'}(\Omega;\Rdsym)$. Let $\{(\ueps,\pieps)\}_{\varepsilon\in(0,1)}$ be a family of weak solutions of the system \eqref{ProbEps} constructed in Lemma~\ref{Thm:Exist}. Then there exists a sequence $\{\varepsilon_k\}_{k=1}^{+\infty}$ such that as $k\rightarrow +\infty$
\begin{alignat*}{2}
\varepsilon_k\to 0,\quad\boldu^{\varepsilon_k}\rightharpoonup \boldu \text{ in }W^{1,p}_0(\Omega;\Rd),\quad
	\pi^{\varepsilon_k}\rightharpoonup \pi \text{ in }L^s(\Omega),
\end{alignat*}
	where $s$ is determined in \eqref{Def:s} and $(\boldu,\pi)$ is a weak solution of the system 
\begin{equation}\label{ProbHom}
	\begin{gathered}
	-\div \left(\widehat{\boldS}(\boldD\boldu)-\boldu\otimes \boldu\right)+\nabla\pi= -\div \boldF\quad\text{ in }\Omega,\quad\div \boldu=0\quad\text{ in }\Omega,\\
	\boldu=0\quad\text{ on }\partial\Omega,\quad\int_\Omega \pi=0.
      \end{gathered}
    \end{equation}
	with $\hat\boldS$ given by \eqref{HSDef}.
\end{Theorem}

We want to emphasize that in our setting we allow that $p<3d/(d+2)$ and so the term $\int_\Omega \boldu\otimes\boldu:\boldD\boldu\dx$ is not defined. In this situation we are not allowed to test weak formulations of the problems \eqref{ProbEps} and \eqref{ProbHom} by their weak solutions. We are not aware of any result on homogenization that would treat this situation.  

Let us outline the structure of the paper. In Section \ref{Sec:Prel} we introduce function spaces appearing in the paper, collect several useful lemmas and show some facts about two-scale convergence and the tensor $\hat\boldS$. The Section \ref{Sec:HomProc} is devoted to the homogenization process.
\section{Preliminaries}\label{Sec:Prel}

The following function spaces appear further: 
$C^\infty_{0,\div}(\Omega)=\left\{\boldu\in C^\infty_0(\Omega;\Rd):\div \boldu=0\text{ in }\Omega\right\}$, $C^\infty_{per}(Y)=\{u\in C^\infty(\Rd): u\ Y\text{-periodic}\}$, $C^\infty_{per,div}(Y)=\{\boldu\in C^\infty_{per}(\Rd): \div \boldu=0 \text{ in } Y\}$, $W^{1,p}_{per}(Y,\Rd)$ is a closure of $\{\boldu\in C^\infty_{per}(Y),\int_Y\boldu=0\}$ in the classical Sobolev norm,\\  $\mathcal{D}\left(\Omega;C^\infty_{per}(Y)\right)$ is the space of smooth functions $u:\Omega\times\Rd\rightarrow \eR$ such that $u(x,\cdot)\in C^\infty_{per}(Y)$ for any $x\in\Omega$ and there is $K\Subset\Omega$ such that for any $x\in\Omega\setminus K$: $u(x,\cdot)=0$ in $\Rd$.

We introduce a closed subspace of $L^{p}(Y;\Rdsym)$ and its annihilator in $L^{p'}(Y;\Rdsym)$ by
\begin{align*}
	&G(Y)=\left\{\boldD\boldw:\boldw\in W^{1,p}_{per}(Y;\Rd), \div\boldw=0 \text{ in } Y\right\},\\
	&G^\bot(Y)=\left\{\boldV^*\in L^{p'}(Y;\Rdsym):\forall \boldV\in G(Y)\ \int_Y \boldV^*(y)\cdot\boldV(y)\dy=0\right\}.
\end{align*}
Note that $C^\infty_{per,div}(Y)$ is dense in $G(Y)$. If we consider the set $\Rdsym$ as a subset of constant functions of $L^{p}(Y;\Rdsym)$ then $\Rdsym\cap G(Y)=\emptyset$.

For the sake of clarity, we recall the meaning of differential operators appearing in the paper. Let us consider $\boldu:\Omega\times Y\rightarrow\Rd$ then
\begin{equation*}
\nabla_x\boldu=\left(\frac{\partial u_i}{\partial x_j}\right)_{i,j=1}^d, \ \div_x\boldu=\sum_{i=1}^d\frac{\partial u_i}{\partial x_i},\ \nabla_y\boldu=\left(\frac{\partial u_i}{\partial y_j}\right)_{i,j=1}^d,\ \div_y\boldu=\sum_{i=1}^d\frac{\partial u_i}{\partial y_i}.
\end{equation*}
We omit the subscript if the function depends on the variable from one domain only. 
Throughout the paper the identity matrix is denoted by $\boldI$, the zero matrix by $\boldO$. The generic constants are denoted by $c$. When circumstances require it, we may also include quantities, on which the constant depend, e.g. $c(d)$ for the dependence on the dimension $d$. If we want to distinguish between different constants in one formula, we utilize subscripts, e.g. $c_1,c_2$ etc. 

Let $M,N$ be open subsets of $\Rd$. $M\Subset N$ means that $M\subset\overline{M}\subset N$, $\overline{M}$ being compact.

\subsection{Auxiliary tools}

\begin{Lemma}{\normalfont (Biting lemma, \cite{BM})}\label{Lem:ChacBit}
Let $E\subset\Rd$ be a bounded domain and $\{v^n\}$ be a sequence of functions bounded in $L^1(E)$. Then there exists a subsequence $\{v^{n_k}\}\subset\{v^n\}$, a function $v\in L^1(E)$ and a sequence of measurable sets $\{E_j\}, E\supseteq E_1\supseteq E_2\supseteq \cdots$ with $|E_j|\rightarrow 0$ as $j\to \infty$ such that for each $j$: $v^{n_k}\rightharpoonup v$ in $L^1(E\setminus E_j)$ as $k\rightarrow\infty$.
\end{Lemma}

\begin{Lemma}{\normalfont (Dunford, \cite[Section III.2 Theorem 15]{DieUh})}\label{Thm:Dun}
Let $\Sigma\in\Rd$ be a measurable set. A subset $M$ of $L^1(\Sigma)$ is relatively weakly compact if and only if it is bounded and uniformly integrable, i.e., for any $\theta>0$ there is $\delta>0$ such that for any $f\in M$ and a measurable $K\subset\Sigma$ with $|K|<\delta$  we have $\int_K |f|<\theta$.
\end{Lemma} 

\begin{Lemma}{\normalfont \cite[Theorem~10.11]{FeiNov}}\label{Lem:BogOp}
Let $\Sigma\subset\Rd$ be a bounded Lipschitz domain, $q\in(1,\infty)$ and denote $L^q_0(\Sigma)=\{h\in L^q:\int_\Sigma h=0\}$. There exists a continuous linear operator $\mathcal{B}:L^q_0(\Sigma)\rightarrow W^{1,q}_0(\Sigma;\Rd)$ such that $\div \mathcal{B}h=h$ for any $h\in L^q_0(\Sigma)$.
\end{Lemma}
\begin{Lemma}{\normalfont \cite[Theorem~10.21]{FeiNov}}\label{Lem:DivCurl}
\def\un{\mathbf{u}^n}
\def\boldu{\mathbf{u}}
\def\vn{\mathbf{v}^n}
\def\boldv{\mathbf{v}}
Let $\Sigma\subset \Rd$ be an open set, $p,q,r>1$. Assume 
\begin{equation*}
	\un\rightharpoonup \boldu\text{ in }L^p(\Sigma;\Rd), \vn\rightharpoonup \boldv\text{ in }L^{q}(\Sigma;\Rd) \mbox{ as $n\rightarrow\infty$ and $\frac{1}{p}+\frac{1}{q}<\frac{1}{r}\leq 1$}.
\end{equation*}
	In addition, let for a certain $s>1$ $\{\div \un\}$ be precompact in $\left(W^{1,s}_0(\Sigma;\Rd)\right)^*$, $\{\curl \vn\}=\{\nabla(\vn)-(\nabla(\vn))^T\}$ be precompact in $\left(W^{1,s}_0(\Sigma;\Rd)^d\right)^*$. Then 
	\begin{equation*}
		\un\cdot\vn\rightharpoonup \boldu\cdot\boldv\text{ in }L^r(\Sigma).
	\end{equation*}
\end{Lemma}
The history of this lemma goes back to the works \cite{Murat} and \cite{Tartar}.

For $f\in L^1(\Rd)$, we define the Hardy-Littlewood maximal function as 
\begin{equation*}
(Mf)(x)=\sup_{r>0}\frac{1}{|B_r(x)|}\int_{B_r(x)}|f(y)|\dy,
\end{equation*}
where $B_r(x)$ stands for a ball having a center at $x$ and radius $r$.

\begin{Lemma}\label{Lem:LTr}
Let $\Omega\subset\Rd$ be open and bounded with a Lipschitz boundary and $\alpha\geq 1$. Then there is $c>0$ such that for any $\boldv\in W^{1,\alpha}_0(\Omega;\Rd)$ and every $\lambda>0$ there is $\boldv^\lambda\in W^{1,\infty}_0(\Omega;\Rd)$ satisfying
\begin{equation}\label{TruncProp}
	\begin{split}
		\|\boldv^{\lambda}\|_{W^{1,\infty}(\Omega)}&\leq \lambda,\\
		|\{x\in\Omega:\boldv(x)\neq\boldv^\lambda(x)\}|&\leq c\frac{\|\boldv\|_{W^{1,\alpha}(\Omega)}^\alpha}{\lambda^\alpha}.
	\end{split}
\end{equation}
\begin{proof}
	The similar assertion, formulated for functions that do not vanish on $\partial\Omega$, appeared in \cite{AcFu}. For our purposes we refer to \cite[Theorem 2.3]{DMS}, which for any $\boldv\in W^{1,\alpha}_0(\Omega;\Rd)$ and any numbers $\theta,\lambda>0$ ensures the existence of $\boldv_{\theta,\sigma}\in W^{1,\infty}_0(\Omega;\Rd)$ such that
	\begin{equation*}
		\|\boldv_{\theta,\sigma}\|_{L^\infty(\Omega)}\leq\theta,\ \|\nabla\boldv_{\theta,\sigma}\|_{L^\infty(\Omega)}\leq c(d,\Omega)\sigma
	\end{equation*}
	and up to a set of Lebesgue measure zero
	\begin{equation*}
		|\{\boldv_{\theta,\sigma}\neq\boldv\}|\subset\Omega\cap\left(\{M(\boldv)>\theta\}\cup\{M(\nabla\boldv)>\sigma\}\right).
	\end{equation*}
	We pick $\boldv\in W^{1,\alpha}_0(\Omega;\Rd)$ and $\lambda>0$. We apply \cite[Theorem 2.3]{DMS} with $\lambda,\frac{\lambda}{c(d,\Omega)}$ and denote $\boldv^\lambda=\boldv_{\lambda,\frac{\lambda}{c(d,\Omega)}}$ to conclude \eqref{TruncProp}$_1$. Moreover, since we have for any $f\in L^\alpha(\Rd)$ and $\sigma>0$
	\begin{equation*}
		|\{|f|>\sigma\}|\leq \int_{\Rd}\left(\frac{|f|}{\sigma}\right)^\alpha= \frac{\|f\|^\alpha_{L^\alpha(\Rd)}}{\sigma^\alpha},
	\end{equation*}
	we obtain for $\alpha>1$ using the strong type estimate for the maximal function, see \cite[Theorem 1]{Stein2}
	\begin{equation*}
			|\{\boldv_{\theta,\sigma}\neq\boldv\}|\leq \frac{\|M(\boldv)\|^\alpha_{L^\alpha(\Rd)}}{\lambda^\alpha}+c\frac{\|M(\nabla\boldv)\|^\alpha_{L^\alpha(\Rd)}}{\lambda^\alpha}\leq c\frac{\|\boldv\|^\alpha_{W^{1,\alpha}(\Omega)}}{\lambda}.
	\end{equation*}
	For $\alpha=1$ the estimate \eqref{TruncProp}$_2$ is a direct consequence of the weak type estimate of the maximal function, see again \cite[Theorem 1]{Stein2}.
\end{proof}
\end{Lemma}

\subsection{Two-scale convergence} 
The following concept of convergence was introduced by Nguetseng in his seminal paper \cite{Nguetseng}: a sequence $\{u^\varepsilon\}$ bounded in $L^2(\Omega)$ is said weakly two-scale convergent to $u^0\in L^2(\Omega\times Y)$ if for any smooth function $\psi:\Rd\times \Rd\rightarrow \eR$, which is $Y-$periodic in the second argument,
\begin{equation}\label{NgWC}
	\lim_{\varepsilon\rightarrow 0}\int_{\Omega}u^\varepsilon(x)\psi\left(x,\frac{x}{\varepsilon}\right)\dx=\int_{\Omega\times Y}u^0(x,y)\psi(x,y)\dx\dy.
\end{equation}
Properties of this notion of convergence were investigated and applied to a number of problems, see \cite{Allaire}, and the concept was also extended to $L^p, p\geq 1$. It was shown later that there is an alternative approach, so called periodic unfolding, for the introduction of the weak two-scale convergence, which allows to represent the two-scale convergence by means of the standard weak convergence in a Lebesgue space on the product $\Omega\times Y$. In the same manner the strong two-scale convergence is introduced. Since it is known that both presented notions of the weak two-scale convergence are equivalent, see \cite{Visintin}, all properties known for the weak two-scale convergence introduced via \eqref{NgWC} hold also for the second approach. We introduce the weak two-scale convergence via periodic unfolding.
\begin{Definition}
	We define functions $n:\eR\rightarrow \eZ$, $r:\eR\rightarrow [0,1)$, $N:\Rd\rightarrow \eZ^d$ and $R:\Rd\rightarrow Y$ as 
	\begin{equation*}
		\begin{split}
			n(x)=\max\{n\in \eZ: n\leq x\},\  r(x)=x-n(x),\\
			N(x)=(n(x_1),\ldots, n(x_d)),\ R(x)=x-N(x).
		\end{split}
	\end{equation*}
	Then we have for any $x\in \Rd,\varepsilon>0$, a two-scale decomposition $x=\varepsilon\left(N\left(\frac{x}{\varepsilon}\right)+R\left(\frac{x}{\varepsilon}\right)\right)$.
	We also define for any $\varepsilon>0$ a two-scale composition function $T_\varepsilon:\Rd\times Y\rightarrow\Rd$ as $T_\varepsilon(x,y)=\varepsilon \left(N\left(\frac{x}{\varepsilon}\right)+y\right)$.
\end{Definition}

\begin{Remark}\label{Rem:TepsUC}
	It follows that $T_\varepsilon(x,y)\rightarrow x$ uniformly in $\Rd\times Y$ as $\varepsilon\rightarrow 0$ since $T_\varepsilon(x,y)=x+\varepsilon\left(y-R\left(\frac{x}{\varepsilon}\right)\right)$.
\end{Remark}

\begin{Definition}\label{Def:TSC}
We say that a sequence of functions $\{v^\varepsilon\}\subset L^r(\Rd)$ 
\begin{enumerate}
	\item converges to $v^0$ weakly two-scale in $L^r(\Rd\times Y)$, $v^\varepsilon\WTSCon v^0$, if $v^\varepsilon\circ T_\varepsilon$ converges to $v^0$ weakly in $L^r(\Rd\times Y)$,
	\item converges to $v^0$ strongly two-scale in $L^r(\Rd\times Y)$, $v^\varepsilon\STSCon v^0$, if $v^\varepsilon\circ T_\varepsilon$ converges to $v^0$ strongly in $L^r(\Rd\times Y)$.
\end{enumerate}

\end{Definition}
	
\begin{Remark}
We define two-scale convergence in $L^r(\Omega\times Y)$ as two-scale convergence in $L^r(\Rd\times Y)$ for functions extended by zero to $\Rd\setminus\Omega$.
\end{Remark}

\begin{Lemma}\label{Lem:Decomp}
	Let $g\in L^1(\Rd;C_{per}(Y))$. Then, for any $\varepsilon>0$, the function $(x,y)\mapsto g(T_\varepsilon(x,y),y)$ is integrable and 
	\begin{equation*}
		\int_{\Rd}g\left(x,\frac{x}{\varepsilon}\right)\dx=\int_{\Rd}\int_Y g(T_\varepsilon(x,y),y)\dy\dx.
	\end{equation*}
	\begin{proof}
		See \cite[Lemma~1.1]{Visintin}
	\end{proof}
\end{Lemma}

\begin{Lemma}\label{Lem:Facts2S}
\def\boldUeps{\mathbf{U}^\varepsilon}
\def\boldU{\mathbf{U}}
\leavevmode
\begin{enumerate}[label=\roman*)]
	\item Let $v\in L^r(\Omega; C_{per} (Y)), r\in[1,\infty)$, $v$ be $Y-$periodic, define $v^\varepsilon(x)=v(\frac x\varepsilon,x)$ for $x\in\Omega$. Then $v^\varepsilon\STSCon v$ in $L^r(\Omega\times Y)$ as $\varepsilon\rightarrow 0$.
	\item Let $v^\varepsilon\WTSCon v^0$ in $L^r(\Omega\times Y)$ then $v^\varepsilon\rightharpoonup \int_Y v^0(\cdot,y)\dy$ in $L^r(\Omega)$.
	\item Let $\{v^\varepsilon\}$ be a bounded sequence in $L^r(\Omega), r\in(1,\infty)$. Then there is $v_0\in L^r(\Omega\times Y)$ and a sequence $\varepsilon_k\to0$ as $k\to+\infty$ such that $v^{\varepsilon_k}\WTSCon v_0$ in $L^r(\Omega\times Y)$ as $k\to+\infty$.
	\item Let $\{v^\varepsilon\}$ converge weakly to $v$ in $W^{1,r}(\Omega), r\in(1,\infty)$ as $\varepsilon\to0$. Then there is $v_0\in L^r(\Omega;W_{per}^{1,r}(Y))$ and a sequence $\varepsilon_k\to0$ as $k\to+\infty$ such that $v^{\varepsilon_k}$ converges strongly to $v$ in $L^r(\Omega)$ and $\nabla v^{\varepsilon_k}$ converges  weakly two-scale to $\nabla_x v+\nabla_y v_0$ in $L^r(\Omega\times Y)^d$ as $k\to+\infty$.
	\item Let $v^\varepsilon\WTSCon v^0$ in $L^r(\Omega\times Y)$ and $w^\varepsilon\STSCon w^0$ in $L^{r'}(\Omega\times Y)$ then $\int_\Omega v^\varepsilon w^\varepsilon\to\int_\Omega\int_Yv^0w^0$.
\end{enumerate}
\begin{proof} 
The equalities  
\begin{equation*}
	v^\varepsilon\circ T_\varepsilon(x,y)=v\left(T_{\varepsilon}(x,y),\frac{T_\varepsilon(x,y)}{\varepsilon}\right)=v(T_{\varepsilon}(x,y),y)
\end{equation*}
hold by definition of $T_\varepsilon$ and $Y-$periodicity of $v$. If $v\in C(\overline{\Omega\times Y})$, Remark~\ref{Rem:TepsUC} immediately implies
\begin{equation*}
	\int_{\Omega\times Y}\left|v(T_\varepsilon(x,y),y)-v(x,y)\right|^r\mathrm{d}x\mathrm{d}y\rightarrow 0\text{ as }\varepsilon\rightarrow 0.
\end{equation*}
For general $v\in L^r(\Omega;C_{per}(Y))$ we need to approximate $v$ by a continuous function and then proceed as in the proof of mean continuity of Lebesgue integrable functions. 

We obtain $(ii)$ if functions independent of $y$-variable are considered in the definition \eqref{NgWC} of the weak convergence in $L^r(\Omega\times Y)$. 

The assertion $(iii)$ is a direct consequence of Lemma~\ref{Lem:Decomp}, the weak compactness of bounded sets in $L^r(\Omega\times Y)$ and Definition \ref{Def:TSC}$_1$. 

For the proof of $(iv)$ with $r=2$ see \cite[Proposition~1.14.~(i)]{Allaire}, the proof for general $r\neq 2$ is analogous.

Statement $(v)$ follows immediately from definition of the weak and strong two-scale convergence and Lemma~\ref{Lem:Facts2S} applied to function $g=v^\varepsilon w^\varepsilon$ independent of $y$, see \cite[Proposition~1.4]{Visintin}.
\end{proof}
\end{Lemma}
\subsection{Properties of the homogenized stress tensor}
In accordance with \cite{Zhikov} we introduce a tensor $\hat{\boldS}:\Rdsym\rightarrow\Rdsym$ as 
\begin{equation}\label{HSDef}
\hat{\boldS}(\boldxi)=\int_Y \boldS(y,\boldxi+\boldV(y))\dy,
\end{equation}
where the function $\boldV$ is a solution of the cell problem: Let $\boldxi\in\Rdsym$ be fixed. We seek $\boldV\in G(Y)$ such that for any $\boldW\in G(Y)$
\begin{equation}\label{CProb}
\int_Y \boldS(y,\boldxi+\boldV(y)):\boldW(y)\dy=0.
\end{equation}
Since $G(Y)$ is reflexive and the tensor $\boldS$ is strictly monotone, the existence and uniqueness of $\boldV$ follows using the theory of monotone operators. In the next section we show that the tensor $\hat\boldS$ arises when the homogenization process $\varepsilon\rightarrow 0_+$ is performed in \eqref{ProbEps}. Properties of $\hat\boldS$ are listed in the following lemma.
\begin{Lemma}
	There are constants $\hat{c}_1,\hat{c}_2>0$ such that for any $\boldxi\in\Rdsym$
	\begin{equation}\label{HSPr}
		\begin{split}
			\hat{\boldS}(\boldxi)\cdot\boldxi\geq {c}_1 |\boldxi|^p-\hat{c}_1,\\
			|\hat{\boldS}(\boldxi))|^{p'}\leq \hat{c}_2 (|\boldxi|^p+1).
		\end{split}
	\end{equation}
	Moreover, $\hat{\boldS}$ is strictly monotone and continuous on $\Rdsym$.
	\begin{proof}
		Let $\boldV$ be a weak solution of the cell problem corresponding to $\boldxi$. Then using \ref{SCoercGrow}, Jensen's inequality and the fact that $\boldV$ is a symmetric gradient of $Y-$periodic function we obtain
		\begin{equation*}
			\begin{split}
				\hat\boldS(\boldxi)\cdot\boldxi&=\int_Y \boldS(y,\boldxi+\boldV(y))\dy\cdot\boldxi=\int_Y \boldS(y,\boldxi+\boldV(y)):(\boldxi+\boldV(y))\dy\\
				&\geq \int_Yc_1|\boldxi+\boldV(y)|^p\dy-\tcj\geq  c_1\left|\boldxi+\int_Y\boldV(y)\dy\right|^p-\tcj=c_1|\boldxi|^p-\tcj,
			\end{split}
		\end{equation*}
		which is \eqref{HSPr}$_1$.
		Using \ref{SCoercGrow}, the H\"older and Young inequalities yields
		\begin{equation*}
			\begin{split}
			|\hat{\boldS}(\boldxi)|^{p'}&=\left|\int_Y\boldS(y,\boldxi+\boldV(y))\dy\right|^{p'}\leq \int_Y|\boldS(y,\boldxi+\boldV(y))|^{p'}\dy\leq \int_Y c_2\big(|\boldxi+\boldV(y)|^p\dy+1\big)\\
			&\leq \int_Y\frac{c_2}{c_1}\boldS(y,\boldxi+\boldV(y))\cdot(\boldxi+\boldV(y))\dy+c_2(1+\frac{\tcj}{c_1})= \frac{c_2}{c_1}\hat\boldS(\boldxi)\cdot\boldxi+c_2(1+\frac{\tcj}{c_1})\\&\leq \frac{1}{2}|\hat\boldS(\boldxi)|^{p'}+\hat{c}_2(|\boldxi|^p+1),
			\end{split}
		\end{equation*}
for suitable (large) $\hat{c}_2>0$. Hence we obtain \eqref{HSPr}$_2$.
		Let $\boldxi^1,\boldxi^2\in\Rdsym$,$\boldxi^1\neq\boldxi^2$ and $\boldV^1,\boldV^2$ be corresponding weak solutions of the cell problem. Then we infer
		\begin{equation*}
			(\hat\boldS(\boldxi^1)-\hat\boldS(\boldxi^2)):(\boldxi^1-\boldxi^2)=\int_Y(\boldS(y,\boldxi^1+\boldV^1(y))-\boldS(y,\boldxi^2+\boldV^2)):(\boldxi^1+\boldV^1-\boldxi^2-\boldV^2)\dy.
		\end{equation*}
		The strict monotonicity of $\hat\boldS$ then follows since the integrand on the right hand side of the latter identity is positive due to \ref{SStrMon}. 

In order to obtain the continuity of $\hat\boldS$ we first show that the mapping $\boldxi\mapsto\boldS(\cdot,\boldxi+\boldV)$, where $\boldV$ is the solution to the corresponding cell problem,  is weakly continuous with values in $L^{p'}(Y;\Rdsym)$. Let us choose $\{\boldxi^n\}_{n=1}^{+\infty}$ such that $\boldxi^n\rightarrow\boldxi$ in $\Rdsym$ as $n\rightarrow+\infty$. Let $\{\boldV^n\}_{n=1}^{+\infty}\subset G(Y)$ be a sequence of solutions of corresponding cell problems and denote $\boldS^n:=\boldS(\cdot,\boldxi^n+\boldV^n)$. Since $\{\boldS^n\}_{n=1}^{+\infty}$ and $\{\boldV^n\}_{n=1}^{+\infty}$ are bounded in $L^{p'}(Y;\Rdsym)$ and $L^{p}(Y;\Rdsym)$ by \ref{SCoercGrow}, they contain weakly convergent subsequences in $L^{p'}(Y;\Rdsym)$ and $L^{p}(Y;\Rdsym)$. Let us assume without loss of generality that 
	\begin{equation}\label{BSNCon}
	\boldV^n\rightharpoonup \boldV^*\text{ in }L^{p}(Y;\Rdsym),\quad	\boldS^n\rightharpoonup \boldS^*\text{ in }L^{p'}(Y;\Rdsym)\text{ as }n\rightarrow+\infty.
	\end{equation}
		We show that $\boldV=\boldV^*$ and $\boldS^*=\boldS(\cdot,\boldxi+\boldV)$. As we have by the definition of the weak solution of the cell problem 
		\begin{equation}\label{eq:weak-cellV}
\forall \boldW\in G(Y):	\int_Y \boldS^n: \boldW=0,
		\end{equation}
		it follows from \eqref{BSNCon} that
		\begin{equation}\label{BSNLimIntPr}
			\lim_{n\rightarrow+\infty}\int_Y \boldS^n : \boldV^n=0=\int_Y\boldS^*:\boldV^*.
		\end{equation}
		Clearly, \ref{SStrMon} implies that 
		\begin{equation}\label{BSNMon}
\forall\boldW\in L^{p'}(Y;\Rdsym):			0\leq\int_Y \left(\boldS^n(y)-\boldS(y,\boldxi^n+\boldW(y))\right):(\boldV^n(y)-\boldW(y))\dy.
		\end{equation}
		We want to perform the limit passage $n\rightarrow+\infty$. Let us observe that for all fixed $\boldW\in L^{p}(Y;\Rdsym)$
		\begin{equation}\label{BSComStrongC}
			\boldS(\cdot,\boldxi^n+\boldW)\rightarrow\boldS(\cdot,\boldxi+\boldW)\text{ in }L^{p'}(Y;\Rdsym)\text{ as }n\rightarrow+\infty.
		\end{equation}
		Indeed, $\boldS(\cdot,\boldxi^n+\boldW)\rightarrow\boldS(\cdot,\boldxi+\boldW)$ a.e. in Y. Moreover, we have for $K\subset Y$ by \ref{SCoercGrow} that
		\begin{equation*}
			\begin{split}
			\int_K |\boldS(y,\boldxi^n+\boldW(y))-\boldS(y,\boldxi+\boldW(y))|^{p'}\dy&\leq \int_K c_2\left(|\boldxi^n+\boldW|^p+|\boldxi+\boldW|^p+2\right)\\
			&\leq c|K|(|\boldxi^n|^p+|\boldxi|^p+2)+c\int_K|\boldW|^p,
			\end{split}
		\end{equation*}
		which implies that $|\boldS(y,\boldxi^n+\boldW)-\boldS(y,\boldxi+\boldW)|^{p'}$ is equiintegrable and the Vitali convergence theorem yields \eqref{BSComStrongC}.
		Hence we obtain from \eqref{BSNMon} using \eqref{BSNLimIntPr} and \eqref{BSComStrongC} that
		\begin{equation*}
\forall\boldW\in L^{p'}(Y;\Rdsym):			0\leq\int_Y \left(\boldS^*(y)-\boldS(y,\boldxi+\boldW(y))\right):(\boldV^*(y)-\boldW(y))\dy.
		\end{equation*}
		Minty's trick gives that $\boldS^*(y)=\boldS(y,\boldxi+\boldV^*(y))$ a.e. in $Y$. Passing to the limit $n\to+\infty$ in \eqref{eq:weak-cellV} we get that $\boldV^*$ is a solution to the cell problem corresponding to $\boldxi$. Since this solution is unique we get $\boldV=\boldV^*$. Up to now we showed that from $\{\boldS^n\}_{n=1}^{+\infty}$ we can extract a subsequence weakly convergent towards $\boldS(\cdot,\boldxi+\boldV)$ in $L^{p'}(Y;\Rdsym)$.  Since this limit is unique, the whole sequence must converge to it. 
		
One easily obtains due to the weak continuity of $\boldxi\mapsto\boldS(\cdot,\boldxi+\boldV)$ 		for any $\boldeta\in\Rdsym$
		\begin{equation*}
			\left(\hat\boldS(\boldxi^n)-\hat\boldS(\boldxi)\right):\boldeta=\int_Y \left(\boldS(y,\boldxi^n+\boldV^n(y))-\boldS(y,\boldxi+\boldV(y))\right):\boldeta\dy\rightarrow 0
		\end{equation*}
as $n\to+\infty$. Since the space $\Rdsym$ is finite dimensional, we have the continuity of $\hat\boldS$.
	\end{proof}
\end{Lemma}

\section{Proof of the main theorem}\label{Sec:HomProc}
Let us outline next steps. First, we observe that for fixed $\varepsilon\in(0,1)$ there is a weak solution of the problem \eqref{ProbEps}, which is bounded uniformly with respect to $\varepsilon$. This uniform boundedness implies the existence of a sequence $\{\boldu^{\varepsilon_k}\}_{k=1}^{+\infty}$ that converges weakly in $W^{1,p}(\Omega;\Rd)$ to a limit $\boldu$. We can also assume that the sequence $\{\boldS^{\varepsilon_k}\}_{k=1}^{+\infty}$ converges weakly in $L^{p'}(\Omega)$ to a limit $\bar\boldS$. To be able to use the Minty trick to identify the limit function $\bar\boldS$ we need to derive 
\begin{equation}\label{LimProdIdent}
	\lim_{k\rightarrow+\infty}\int_\Omega\boldS^{\varepsilon_k}:\boldD \boldu^{\varepsilon_k}=\int_\Omega \bar\boldS:\boldD \boldu.
\end{equation}
The assumption $p\geq \frac{2d}{d+2}$ ensures the precompactness of the term $\ueps\otimes\ueps$ in $L^{q}(\Omega;\Rdsym)$ for some $q$ that might be less than $p'$. Then one cannot test the weak formulation involving the limit $\bar\boldS$ with the limit $\boldu$ to obtain \eqref{LimProdIdent}. Actually in this situation we are not allowed to test \eqref{ProbEps} with $\ueps$ as well. To overcome this inconvenience we decompose the pressure $\pieps$ into three parts. The first one corresponds to $\Seps$ and is bounded in $L^{p'}$, the second one corresponds to $\boldF+\ueps\otimes\ueps$ and is precompact in $L^q$ for any $q\in(1,s)$, where $s$ might be less than $p'$ and is determined by \eqref{Def:s} and the last part is harmonic. Then we employ the div-curl lemma to obtain the identity of the type \eqref{LimProdIdent}. In fact, we show \eqref{LimProdIdent} for certain subsets of $\Omega$ determined by Biting lemma \ref{Lem:ChacBit}.

Let us begin with the existence of a weak solution of \eqref{ProbEps}.
\begin{Lemma}\label{Thm:Exist}
Let $\Omega\subset\Rd$ be a bounded Lipschitz domain, $\varepsilon>0$ be fixed, $\boldF\in L^{p'}(\Omega;\Rdsym)$, $p>2d/(d+2)$, Assumption~\ref{Ass:S} be fulfilled and $s$ be determined by
\begin{equation}\label{Def:s}
	s=
	\begin{cases}
		\min\left\{\frac{dp}{2(d-p)},p'\right\}&p<d,\\
		p' & p\geq d.
	\end{cases}
\end{equation}
Then there exists a weak solution $(\ueps,\pi^\varepsilon)$ of \eqref{ProbEps}, which is a pair $(\ueps,\pieps)\in W^{1,p}_{0,\div}(\Omega;\Rd) \times L^s(\Omega)$ such that for any $\boldw\in C^\infty_{0}(\Omega;\Rd)$
	\begin{equation}\label{StacWF}
		\int_\Omega (\Seps-\ueps\otimes \ueps-\pieps\boldI):\boldD\boldw=\int_\Omega \boldF: \boldD\boldw.
	\end{equation}

 Moreover, there is $c>0$ independent of $\varepsilon$ such that
	\begin{equation}\label{AprEst}
		\begin{split}
			\|\boldD\ueps\|_{L^p(\Omega)}&\leq c,\\
			\|\pi^\varepsilon\|_{L^s(\Omega)}&\leq c.
		\end{split}
	\end{equation}
\begin{proof}
	Due to Assumptions~\ref{Ass:S} we can adopt the technique used for the proof in \cite[Theorem~3.1]{DMS}.
\end{proof}
\end{Lemma}

At this moment we decompose the pressure into a part that is bounded in $L^{p'}$, a part that is precompact but in a bigger space and a harmonic part.
\begin{Lemma}\label{Lem:PressComp}
Let $s$ be given by \eqref{Def:s} and the functions $\ueps,\pi^\varepsilon,\Seps,\boldF$ be extended by zero on $\Rd\setminus\Omega$ for $\varepsilon>0$. Let functions $\pi^{\varepsilon,1}\in L^{p'}(\Rd),\pi^{\varepsilon,2}\in L^\frac{p^*}{2}(\Rd),\pi^{\varepsilon,3}\in L^{p'}(\Omega)$ be defined as
\newcommand{\NN}{{\cal N}}
\begin{align*}
\pi^{\varepsilon,1}&=\div\div \NN\left(\Seps\right),\\
\pi^{\varepsilon,2}&=-\div\div \NN\left(\boldF+\ueps\otimes\ueps\right),\\
\pi^{\varepsilon,3}&=\pi^\varepsilon-\pi^{\varepsilon,1}-\pi^{\varepsilon,2}.
\end{align*}
Here $\NN\left(\Seps\right)$ denotes the componentwise Newton potential of $\Seps$ and $p^*=dp/(d-p)$ if $d>2$ and $p^*>1$ if $d=2$.
Then 
	\begin{equation}\label{PressProp}
	\begin{split}
		&\{\pi^{\varepsilon,1}\}\text{ is bounded in } L^{p'}(\Rd),\\
		&\{\pi^{\varepsilon,2}\}\text{ is precompact in } L^{q}(\Rd)\text{ for any }q\in[1,s),\\
		&\{\pi^{\varepsilon,3}\}\text{ is precompact in } L^{p'}(O)\text{ for any }O\Subset\Omega.
		\end{split}
	\end{equation}
\begin{proof}
Applying the theory of Calderon-Zygmund operators, see \cite[Section 6.3]{DHHR}, yields the estimates
\begin{equation}\label{PiEps1Est}
	\|\pi^{\varepsilon,1}\|_{L^{p'}(\Rd)}\leq c\|\Seps\|_{L^{p'}(\Rd)},\quad \|\pi^{\varepsilon,2}\|_{L^{s}(\Rd)}\leq c,
\end{equation} 
and the precompactness of $\{\pi^{\varepsilon,2}\}$ in $L^q(\Rd), q\in [1,s)$ since $\{\boldF+\ueps\otimes\ueps\}$ is precompact in $L^q(\Rd;\Rdsym)$ by \eqref{WConvcs}$_2$. It follows from \eqref{StacWF} and \eqref{AprEst} that $\{\pi^{\varepsilon,3}\}_{\varepsilon\in(0,1)}$ are harmonic functions in $\Omega$ and bounded in $L^s(\Omega)$. Hence we can extract from $\{\pi^{\varepsilon,3}\}_{\varepsilon\in(0,1)}$ a subsequence that converges uniformly in any $O\Subset\Omega$. Thus $\{\pi^{\varepsilon,3}\}_{\varepsilon\in(0,1)}$ is precompact in $L^{p'}(O)$.
\end{proof}
\end{Lemma}

We want to find a sequence $\{\varepsilon_k\}_{k=1}^{+\infty}\subset (0,1)$ such that $\varepsilon_k\to0$ as $k\to+\infty$ so that sequences of functions $\{\boldS^{\varepsilon_k}\}_{k=1}^{+\infty}$, $\{\boldu^{\varepsilon_k}\}_{k=1}^{+\infty}$ and $\{\pi^{\varepsilon_k}\}_{k=1}^{+\infty}$ has some additional properties. To abbreviate the notation we will write $F^k$ for $F^{\varepsilon_k}$.

We start with extracting convergent subsequences.

\begin{Lemma}\label{lem:2s}
Let $s$ be given by \eqref{Def:s}. For any $N\subset(0,1)$, $0\in\partial N$ there is a sequence $\{\varepsilon_k\}_{k=1}^{+\infty}\subset N$ such that $\varepsilon_k\to0$ as $k\to+\infty$ and functions $\boldu\in W^{1,p}_{0,\div}(\Omega;\Rd)$, $\boldu^0\in L^p\left(\Omega;W^{1,p}_{per}(Y)^d\right)$, $\overline{\boldS^0}\in L^{p'}(\Omega\times Y;\Rdsym)$, $\overline\pi\in L^s(\Omega\times Y)$ and $\overline{\pi^1}\in L^{p'}(\Omega\times Y)$ such that as $k\to+\infty$ 
	\begin{alignat}{2}
		\boldD\uk&\WTSCon \boldD\boldu+\boldD_y\boldu^0 &&\text{ in }L^p(\Omega\times Y;\Rdsym),\label{WTSConGrU}\\
		\Sepsk&\WTSCon \overline{\boldS^0} &&\text{ in }L^{p'}(\Omega\times Y;\Rdsym),\label{WTSConS}\\
		\pi^{k}&\WTSCon \overline{\pi} &&\text{ in }L^{s}(\Omega\times Y),\label{WTSConP}\\
		\pi^{k,1}&\WTSCon \overline{\pi^1} &&\text{ in }L^{p'}(\Omega\times Y),\label{WTSConP1}
	\end{alignat}
and
	\begin{equation}\label{WConvcs}
		\begin{alignedat}{2}
			\uk &\rightharpoonup \boldu &&\text{ in }W^{1,p}_0(\Omega;\Rd),\\
			\uk &\rightarrow \boldu &&\text{ in }L^{p^*}(\Omega;\Rd),\\
			\pi^k &\rightharpoonup \pi:=\int_Y\overline{\pi}(\cdot,y)\dy &&\text{ in }L^s(\Omega),\\
		\pi^{k,1}&\rightharpoonup \pi^1:=\int_Y\overline{\pi^1}(\cdot,y)\dy&&\text{ in }L^{p'}(\Omega),\\
			\Sepsk &\rightharpoonup \overline{\boldS}:=\int_Y\overline{\boldS^0}(\cdot,y)\dy &&\text{ in }L^{p'}(\Omega;\Rdsym).
		\end{alignedat}
	\end{equation}

	Moreover, the limit functions satisfy 
	\begin{align}
			&\text{for almost all }x\in \Omega:\ \boldD_y\boldu^0(x,\cdot)\in G(Y),\label{VInG}\\
			&\text{for almost all }x\in \Omega:\ \overline{\boldS^0}(x,\cdot)\in G^\bot(Y),\label{sInGB}\\
			&\text{for almost all }x\in \Omega:\ \overline{\pi^1}(x,\cdot)\boldI\in G^\bot(Y),\label{P1InGB}
	\end{align}	
	and for any $\boldw\in C^\infty_{0}(\Omega)$
	\begin{equation}\label{WFLim}
		\int_\Omega (\overline{\boldS}-\boldu\otimes \boldu-\pi\boldI):\boldD\boldw =\int_\Omega \boldF:\boldD\boldw.
	\end{equation}

	\begin{proof} In the proof we subsequently extract a subsequences. We will not explicitely refer to this fact.

The convergences in \eqref{WConvcs} follow in a standard way from \eqref{AprEst}, \eqref{PressProp}, Sobolev embedding theorem. Validity of \eqref{WFLim} follows from \eqref{WConvcs} and \eqref{StacWF}.

	Convergence \eqref{WConvcs}$_1$ and Lemma~\ref{Lem:Facts2S}~$(iv)$ imply \eqref{WTSConGrU}. Statements \eqref{WTSConS}, \eqref{WTSConP}, \eqref{WTSConP1} and identification of weak limits in \eqref{WConvcs} follow from \eqref{AprEst}, \eqref{PressProp}$_1$, the assumption on growth of $\boldS$ and Lemma~\ref{Lem:Facts2S}~$(ii)$ and $(iii)$.

Let us show \eqref{VInG}. The convergence \eqref{WTSConGrU} means that for any $\bpsi\in \mathcal{D}\left(\Omega;C^\infty_{per}(Y)^{d\times d}\right)$
\begin{equation}\label{WTSCGrUeps}
	\lim_{k\rightarrow +\infty}\int_{\Omega}\boldD\uk(x):\bpsi\left(x,\frac{x}{\varepsilon_k}\right)\dx=\int_{\Omega}\int_Y\left(\boldD\boldu(x)+\boldD_y\boldu^0(x,y)\right):\bpsi(x,y)\dx\dy.
\end{equation}
We pick $a\in \mathcal{D}(\Omega), b\in C^\infty_{per}(Y)$ and put $\bpsi(x,y)=a(x)b(y)\boldI$ in \eqref{WTSCGrUeps}. Obviously, we get using the weak convergence of $\{\uk\}$ in $W^{1,p}_0(\Omega;\Rd)$
\begin{equation*}
\begin{split}
	0&=\lim_{k\rightarrow +\infty}\int_{\Omega}\div\uk(x)a(x)b\left(\frac{x}{\varepsilon_k}\right)\dx=\lim_{k\rightarrow +\infty}\int_\Omega\boldD\uk(x)a(x)b\left(\frac{x}{\varepsilon_k}\right):\boldI\dx\\&=\int_{\Omega}\int_Y\left(\boldD\boldu(x)+\boldD_y\boldu^0(x,y)\right)a(x)b(y):\boldI\dy\dx\\&=\int_{\Omega}\div\boldu(x)a(x)\int_Yb(y)\dy\dx+\int_\Omega\int_Y\div_y\boldu^0(x,y)b(y)\dy\ a(x)\dx\\&=\int_\Omega\int_Y\div_y\boldu^0(x,y)b(y)\dy\ a(x)\dx.
	\end{split}
\end{equation*}
Hence for a.a. $x\in\Omega$ $\div_y\boldu^0(x,\cdot)=0$ a.e. in $Y$, i.e. we conclude \eqref{VInG}.

We show that for any $\sigma\in C^\infty_0(\Omega)$ and $\boldh\in C^\infty_{per,\div}(Y;\Rd)$
\begin{equation}\label{SOrt}
	\int_\Omega\int_Y \overline{\boldS^0}(x,y):\boldD\boldh(y)\dy\sigma(x)\dx=0.
\end{equation}
	Since $\varepsilon_k\sigma(x)\boldh\left(\frac{x}{\varepsilon_k}\right)$ is not solenoidal, the correction $\bk(x)=\mathcal{B}\left(\varepsilon_k \boldh\left(\frac{x}{\varepsilon_k}\right)\nabla\sigma(x)\right)$ that satisfies
	\begin{equation*}
		\begin{split}
			\div \bk(x)&=\varepsilon_k \boldh(\frac{x}{\varepsilon_k})\nabla \sigma(x)\text{ in }x\in\Omega, \quad\bk=0\text{ on }\partial\Omega,\\
			\|\nabla \bk\|_{L^\gamma(\Omega)}&\leq c(\gamma,\|\boldh\|_{L^\infty(Y)},\|\nabla\sigma\|_{L^\infty(\Omega)})\varepsilon_k
		\end{split}
	\end{equation*}
with an arbitrary $\gamma\in(1,\infty)$, is introduced to allow using $\varepsilon_k\sigma(x)\boldh\left(\frac{x}{\varepsilon_k}\right)-\bk$ as a test function in \eqref{StacWF}. We note that the existence of $\bk$ is ensured by Lemma~\ref{Lem:BogOp}. Then we employ convergences as $k\to+\infty$
\begin{alignat*}{2}
	\bk&\rightarrow 0&&\text{ in }L^\gamma(\Omega;\Rd),\\
	\boldD\bk&\rightarrow 0 &&\text{ in }L^\gamma(\Omega;\Rdsym),\\
	\sigma\uk\otimes\uk &\rightarrow\sigma\boldu\otimes\boldu &&\text{ in }L^1(\Omega;\Rdsym),\\
	\Sepsk&\WTSCon \overline{\boldS^0}&&\text{ in }L^{p'}(\Omega;\Rdsym),
\end{alignat*}
to obtain from \eqref{StacWF} by \eqref{NgWC} and Lemma~\ref{Lem:Facts2S}$_1$ that 
\begin{equation*}
	\int_\Omega\int_Y (\overline{\boldS^0}(x,y)-\boldu(x)\otimes \boldu(x)-\boldF(x)):\boldD_y \boldh(y)\dy\sigma(x)\dx=0.
\end{equation*}
Hence \eqref{SOrt} and thus \eqref{sInGB} follow due to an obvious fact $\int_Y\boldD_y\boldh(y)=0$.

	Finally, we infer that for any $\boldW=\boldD\boldw\in G(Y)$ and almost all $x\in\Omega$
		\begin{equation*}
			\int_Y \overline{\pi^1}(x,y)\boldI:\boldW(y)\dy=\int_Y \overline{\pi^1}(x,y)\div_y \boldw(y)\dy=0,
		\end{equation*}
		which concludes \eqref{P1InGB}.
	\end{proof}
\end{Lemma}

Further we also utilize the following lemma concerning the equiintegrability property of sequences $\{\boldS^k\}_{k=1}^\infty$, $\{\pi^{k,1}\}_{k=1}^\infty$.
\begin{Lemma}\label{Lem:Equiint}
For any $N\subset(0,1)$, $0\in\partial N$ there is a sequence $\{\varepsilon_k\}_{k=1}^{+\infty}\subset N$ such that $\varepsilon_k\to0$ as $k\to+\infty$ and a sequence of measurable sets $\Omega_1\subset\Omega_2\subset\cdots\subset\Omega_n\subset\cdots\subset\Omega$ with $|\Omega\setminus\Omega_n|\rightarrow 0$ as $n\to +\infty$ such that for any $n\in\eN$ and $\theta>0$ there is $\delta>0$ such that for any $k\in\eN$ and $K\subset\Omega_n$ with $|K|<\delta$
\begin{equation}\label{NormSmall}
\|\boldS^k\|_{L^{p'}(K)}+\|\pi^{k,1}\|_{L^{p'}(K)}<2\theta^\frac{1}{p'}.
\end{equation}
\begin{proof}
	Let us consider an arbitrary sequence $\{\varepsilon_k\}_{k=1}^{+\infty}\subset N$, $\varepsilon_k\to0$ as $k\to+\infty$ and denote $G^k=|\boldS^k|^{p'}+|\pi^{k,1}|^{p'}$. The apriori estimate \eqref{AprEst}$_1$, the growth condition on $\boldS$ and \eqref{PressProp}$_1$ imply the boundedness of $\{G^k\}_{k=1}^\infty$ in $L^{1}(\Omega)$. The application of Chacon's biting lemma \ref{Lem:ChacBit} on $\{G^k\}_{k=1}^\infty$ yields the existence of sets $\Omega_n\subset\Omega$ with $|\Omega\setminus\Omega_n|\rightarrow 0$ as $n\rightarrow\infty$ and the existence of a subsequence $\{G^k\}_{k=1}^\infty$ (that will not be relabeled) and a function $G\in L^1(\Omega)$ such that $G^k\rightharpoonup G$ in $L^1(\Omega_n)$ as $k\rightarrow\infty$. According to Dunford theorem \ref{Thm:Dun} we obtain the equiintegrability of $\{G^k\}_{k=1}^\infty$ on $\Omega_n$, i.e., for any $\theta>0$ there is $\delta>0$ such that for any $k\in\eN$ and $K\subset\Omega_n$ with $|K|<\delta$ we have
	\begin{equation*}
		\int_K|\boldS^k|^p+|\pi^{k,1}|^p<\theta,
	\end{equation*}
	which implies \eqref{NormSmall}.
\end{proof}
\end{Lemma}

From now on we assume that a sequence $\{\varepsilon_k\}_{k=1}^{+\infty}\subset N$, $\varepsilon_k\to0$ as $k\to+\infty$ is chosen in such a way that all conclusion of Lemmas~\ref{Lem:PressComp}, \ref{lem:2s} and \ref{Lem:Equiint} hold. In particular we fix the corresponding sequence $\{\Omega_n\}_{n=1}^{+\infty}$ of sets from Lemma~\ref{Lem:Equiint}. The rest of the paper is devoted to finding for this particular sequence the relation between $\overline{\boldS}$ and $\boldD\boldu$. 


In the following lemma we construct for any element of a subsequence of $\{\boldu^k\}_{k=1}^{+\infty}$ a sequence $\{\boldu^{k,\lambda}\}_{\lambda=1}^{+\infty}$, whose elements have the symmetric gradient bounded in $L^\infty(\Omega)$ independently of $k$. Hence a subsequence that converges weakly-star in $L^\infty(\Omega)$ can be selected. Moreover, limit functions form a sequence that contains a weakly convergent subsequence in $L^p(\Omega)$.
\begin{Lemma}\label{Lem:LipsTr}
\def\uepsl{\mathbf{u}^{k,\lambda}}
There is $c>0$ and a subsequence of $\{\varepsilon_k\}_{k=1}^{+\infty}$ (that will not be relabeled) such that 
	\begin{gather}
		\forall k,\lambda\in\eN:\|\boldD\uepsl\|_{L^p(\Omega)}\leq c,\label{PEpsLAmSymGrEst}\\
		\forall \lambda\in\eN:\boldD \uepsl \xrightharpoonup{}^* \boldD\boldu^\lambda\text{ as $k\to+\infty$  in }L^{\infty}(\Omega;\Rd),\label{LIWSConv}
        \end{gather}
where we denoted by $\uepsl$ functions constructed to $\uk$ by Lemma~\ref{Lem:LTr}.
Moreover, a subsequence $\{\boldu^{\lambda_l}\}_{l=1}^{+\infty}$ can be selected such that
	\begin{equation}
		\boldD \boldu^{\lambda_k} \rightharpoonup\boldD\boldu\text{ as $k\to+\infty$ in }L^p(\Omega;\Rdsym),\label{WCTr}.
	\end{equation}
\begin{proof}
	The application of Lemma~\ref{Lem:LTr} to the sequence $\{\uk\}$ yields the existence of functions $\uepsl\in W^{1,\infty}(\Omega;\Rd)$, $k,\lambda\in\eN$ satisfying
	\begin{equation}\label{SeqTrPr}
			\|\uepsl\|_{W^{1,\infty}(\Omega)}\leq\lambda,\quad
			|\{x\in\Omega:\uk(x)\neq\uepsl(x)\}|\leq c\frac{\|\uk\|_{W^{1,p}(\Omega)}^p}{\lambda^p}.
	\end{equation}
	Utilizing \eqref{SeqTrPr}, Friedrichs and Korn's inequalities, we obtain
	\begin{equation*}
		\begin{split}
		\int_{\Omega}|\boldD\uepsl|^p&=\int_{\{\uk=\uepsl\}}|\boldD\uepsl|^p+\int_{\{\uk\neq\uepsl\}}|\boldD\uepsl|^p\\
		&\leq\int_{\Omega}|\boldD\uk|^p+\lambda^p|\{x\in\Omega:\uk(x)\neq\uepsl(x)\}|\\
		&\leq c\|\boldD\uk\|_{L^p(\Omega)}^p,
		\end{split}
	\end{equation*}
	which implies \eqref{PEpsLAmSymGrEst} due to \eqref{AprEst}. 
	
	The convergence \eqref{LIWSConv} follows from \eqref{SeqTrPr}$_1$ by a diagonal procedure. Moreover, the estimate \eqref{PEpsLAmSymGrEst}, \eqref{LIWSConv} and the weak lower semicontinuity of the $L^p-$norm imply the existence of a positive constant $c$ such that
	\begin{equation*}
	\forall\lambda\in\eN:	\|\boldD\boldu^\lambda\|_{L^p(\Omega)}\leq c.
	\end{equation*}
	Hence we can pick a function $\tilde{\boldu}\in W^{1,p}_0(\Omega;\Rd)$ and a subsequence $\{\lambda_l\}_{l=1}^{+\infty}$ such that 
	\begin{equation*}
		\boldu^{\lambda_l}\rightharpoonup \tilde{\boldu}\text{ as $l\to+\infty$ in }W^{1,p}(\Omega;\Rd).
	\end{equation*}
	It remains to show $\tilde{\boldu}=\boldu$. Using the boundedness of the sequences $\{\uk\},\{\uepsl\}$ in $W^{1,p}(\Omega;\Rd)$ and the estimate \eqref{SeqTrPr}$_2$, we obtain
	\begin{equation*}
		\int_{\Omega}|\uepsl-\uk|=\int_{\{\uepsl\neq\ueps\}}|\uepsl-\uk|\leq\|\uepsl-\uk\|_{L^p(\Omega)}|\{\uepsl\neq\uk\}|^\frac{1}{p'}\leq \frac{c}{\lambda^{p-1}}.
	\end{equation*}
	Moreover, the compact embedding $W^{1,p}(\Omega;\Rd){\hookrightarrow}L^1(\Omega;\Rd)$ implies
	\begin{equation*}
		\|\boldu^\lambda-\boldu\|_{L^1(\Omega)}=\lim_{k\rightarrow +\infty}\|\uepsl-\uk\|_{L^1(\Omega)}.
	\end{equation*}
	Therefore $\boldu^\lambda\rightarrow\boldu$ in $L^1(\Omega)$ and we conclude $\tilde{\boldu}=\boldu$ a.e. in $\Omega$. 
\end{proof}
\end{Lemma}

In the rest of the paper we denote for any $k,l\in\eN$ the function $\boldu^{k,l}:=\boldu^{\varepsilon_k,\lambda_l}$, where $\{\lambda_l\}$ and $\{\varepsilon_k\}$ are sequences constructed in Lemma~\ref{Lem:LipsTr}. Now, we are prepared to show that for certain subsets $\tilde\Omega_n$ of $\Omega$ we can identify $\lim_{k\rightarrow\infty}\int_{\tilde{\Omega}_n}\Sepsk:\boldD\uk\dx$.

\begin{Lemma}\label{Lem:LimIntProd}
\def\uepsl{\mathbf{u}^{\varepsilon,\lambda}}
	Let $O\Subset\Omega$ be arbitrary open and denote $\tilde{\Omega}_n=\Omega_n\cap O$. Then for each $n\in \eN$
	\begin{equation}\label{ProdWConv}
		\lim_{k\rightarrow +\infty}\int_{\tilde{\Omega}_n}\Sepsk:\boldD\uk=\int_{\tilde{\Omega}_n}\overline\boldS :\boldD\boldu.
	\end{equation}
	\begin{proof}
For fixed $n\in\eN$ and any $k,l\in\eN$ we decompose using the solenoidality of $\uk$ 
		\begin{equation*}
			\begin{split}
			\int_{\tilde{\Omega}_n}\Sepsk: \boldD\uk=&\int_{\tilde{\Omega}_n}(\Sepsk-\piepsok\boldI) :\boldD\uk=\int_{\tilde{\Omega}_n}(\Sepsk-\piepsok\boldI):\boldD\left(\uk-\uepskl\right)\\&+\int_{\tilde{\Omega}_n}(\Sepsk-\piepsok\boldI):\boldD\uepskl=I^{k,l}+II^{k,l}.
			\end{split}
		\end{equation*}

We want to perform the limit passage $k\to+\infty$ and then $l\to+\infty$ in both terms on the right hand side of the latter equality. We denote $\tilde{\Omega}_n^{k,l}=\tilde{\Omega}_n\cap\{\uk\neq\uepskl\}$ and estimate using H\"older's inequality, \eqref{PiEps1Est}, \eqref{AprEst}$_1$ and \eqref{PEpsLAmSymGrEst}
		\begin{equation*}
			\begin{split}
				|I^{k,l}|\leq c\|\Sepsk-\piepsok\boldI\|_{L^{p'}(\tilde{\Omega}_n^{k,l})}\|\boldD(\uk-\uepskl)\|_{L^{p}(\tilde{\Omega}_n^{k,l})}\leq c\left(\|\Sepsk\|_{L^{p'}(\tilde{\Omega}_n^{k,l})}+\|\piepsok\|_{L^{p'}(\tilde{\Omega}_n^{k,l})}\right).
			\end{split}
		\end{equation*}
		As $|\tilde\Omega_n^{k,l}|\leq c\lambda_l^{-p}$ by \eqref{SeqTrPr}$_2$, we get by Lemma~\ref{Lem:Equiint}
		that for any $\theta>0$ there exists $l_0\in\eN$ such that for any $l>l_0$ and $k\in\eN$ we have $|I^{k,l}|<\theta$ and therefore
		\begin{equation*}
			\lim_{l\rightarrow+\infty}\lim_{k\to+\infty}I^{k,l}=\lim_{k\to+\infty}\lim_{l\to+\infty}I^{k,l}=0.
		\end{equation*} 
Note that for this estimate it is essential that $\{\pi^{k,1}\}$ is bounded in $L^{p'}(\Omega)$. The terms $\pi^{k,2}$ and $\pi^{k,3}$ cannot be included to $I^{k,l}$.

		For the limit passage $k\to+\infty$ in $II^{k,l}$ we employ Lemma~\ref{Lem:DivCurl}. Let us pick $q\in\left(1,s\right)$, where $s$ is determined by \eqref{Def:s}. We have for any $\boldw\in W^{1,q'}_0(O;\Rd)$	in the sense of distributions
		\begin{equation}\label{eq:2cross}
		\left\langle \div(\Sepsk-\piepsok\boldI),\boldw\right\rangle=-\int_{O}\left(\boldF+\uk\otimes\uk+(\pi^{k,2}+\pi^{k,3})\boldI\right):\boldD\boldw.
		\end{equation}
		It follows from Lemma~\ref{Lem:PressComp} that $\{\boldF+\uk\otimes\uk+(\pi^{k,2}+\pi^{k,3})\boldI\}$ is precompact in $L^q(O;\Rdsym)$. Therefore we obtain that $\{\div(\Sepsk+\piepsok\boldI)\}$ is precompact in $W^{-1,q}(O;\Rd)$. Here it is necessary that part of the pressure corresponding to $\boldS^{k}$, i.e. $\pi^{k,1}$, that is not precompact in any Lebesgue space, does not appear on the right hand side of \eqref{eq:2cross}. We observe that $\curl(\nabla\uepskl)=0$. Then Lemma~\ref{Lem:DivCurl} and the convergences \eqref{WConvcs}$_{4,5}$ and \eqref{LIWSConv} imply
		\begin{equation*}
	(\Sepsk-\piepsok\boldI):\boldD\uepskl=			(\Sepsk-\piepsok\boldI):\nabla\uepskl
\rightharpoonup (\overline\boldS-\pi^1\boldI):\nabla\boldu^l=
(\overline\boldS-\pi^1\boldI):\boldD\boldu^l
		\end{equation*}
in $L^r(O)$ for any $r>1$ 
		as $k\to+\infty$. Hence we deduce using \eqref{WCTr} and the solenoidality of $\boldu$
		\begin{equation*}
\begin{aligned}
			\lim_{l\to+\infty}\lim_{k\to+\infty}II^{k,l}&=\lim_{l\to+\infty}\lim_{k\to+\infty}\int_{O} (\Sepsk-\piepsok\boldI):\boldD \uepskl\chi_{\tilde{\Omega}_n}\\&=\lim_{l\to+\infty}\int_{O} (\overline\boldS-\pi^1\boldI):\boldD \boldu^l\chi_{\tilde\Omega_n}=\int_{\tilde{\Omega}_n} \overline{\boldS}:\boldD \boldu.
                      \end{aligned}
                    \end{equation*}
	\end{proof}
\end{Lemma}

Having all preliminary claims shown we justify the limit passage $\varepsilon\rightarrow 0$ in the weak formulation of \eqref{ProbEps}.
\begin{proof}[Proof of Theorem~\ref{Thm:Homog}]
\def\boldU{\mathbf{U}}
	It remains to show the relation 
	\begin{equation}\label{eq:slovak}
	\overline{\boldS^0}(x,y)=\boldS(x,\boldD\boldu(x)+\boldD_y\boldu^0(x,y))\text{ for  almost all }x\in\Omega, y\in Y.
	\end{equation}
This equality namely immediately implies that $\boldD_y\boldu^0(x,\cdot)$ is the solution of the cell problem \eqref{CProb} with $\boldxi=\boldD\boldu(x)$ for a.a. $x\in\Omega$ by \eqref{VInG} and \eqref{sInGB}. Consequently, integrating \eqref{eq:slovak} over $Y$ we obtain $\overline{\boldS}(x)=\int_Y\boldS(x,\boldD\boldu+\boldD_y\boldu^0)\dy=\hat{\boldS}(\boldD\boldu)$ and \eqref{ProbHom} holds.

Finally, we prove \eqref{eq:slovak}.	We fix $n\in\eN$, a corresponding $\Omega_n$ from Lemma~\ref{Lem:Equiint} and $O\Subset\Omega$. Keeping the notation of Lemma~\ref{Lem:LimIntProd}, using \eqref{WConvcs}, \eqref{VInG} and \eqref{sInGB}, it follows from \eqref{ProdWConv} that 
\begin{equation}\label{LimProd2S}
	\lim_{k\to+\infty}\int_{\tilde\Omega_n} \Sepsk : \boldD\uk=\int_{\tilde\Omega_n}\int_Y \overline{\boldS^0}:(\boldD\boldu+\boldD_y\boldu^0).
\end{equation}
We choose $\boldU\in L^p(\tilde\Omega_n;C_{per}(Y;\Rdsym))$. The monotonicity of $\boldS$ implies
	\begin{equation*}
		\begin{split}
		0\leq& \int_{\tilde\Omega_n}\left(\Sepsk(x)-\boldS\left({x}{\varepsilon_k^{-1}},\boldU\left(x,{x}{\varepsilon_k^{-1}}\right)\right)\right):\left(\boldD\uk(x)-\boldU\left(x,{x}{\varepsilon_k^{-1}}\right)\right)\dx\\=&\int_{\tilde\Omega_n}\Sepsk(x): \boldD\uk(x)\dx-\int_{\tilde\Omega_n}\boldS\left({x}{\varepsilon_k^{-1}},\boldU\left(x,{x}{\varepsilon_k^{-1}}\right)\right):\boldD\uk(x)\dx\\
		&-\int_{\tilde\Omega_n}\Sepsk(x):\boldU\left(x,{x}{\varepsilon_k^{-1}}\right)\dx+\int_{\tilde\Omega_n}\boldS\left({x}{\varepsilon_k^{-1}},\boldU\left(x,{x}{\varepsilon_k^{-1}}\right)\right):\boldU\left(x,{x}{\varepsilon_k^{-1}}\right)\dx\\=&I^k-II^k-III^k+IV^k.
		\end{split}
	\end{equation*}
	We want to pass to the limit as $k\to+\infty$ in $I^k,II^k,III^k,IV^k$. We use \eqref{LimProd2S} for the passage in $I^k$.
	Applying Lemma~\ref{Lem:Facts2S}~$(i)$ to $\boldS(y,\boldU(x,y))$ and $\boldU$  yields

	\begin{equation*}
		\begin{split}
			\boldS\left({x}{\varepsilon_k^{-1}},\boldU\left(x,{x}{\varepsilon_k^{-1}}\right)\right)&\STSCon \boldS(y,\boldU(x,y)) \text{ in }L^{p'}(\Omega\times Y;\Rdsym),\\
			\boldU\left(x,{x}{\varepsilon_k^{-1}}\right)&\STSCon\boldU(x,y)\text{ in }L^{p}(\Omega\times Y;\Rdsym)
		\end{split}
	\end{equation*}
 as $k\to+\infty$.	Employing these convergences and \eqref{WTSConGrU} we infer 
	\begin{equation*}
		\begin{split}
			&\lim_{k\to+\infty} II^k=\int_{\tilde\Omega_n}\int_Y \boldS(y,\boldU(x,y)):(\boldD\boldu(x)+\boldD_y\boldu^0(x,y)\dy\dx, \\
			&\lim_{k\to+\infty} III^k=\int_{\tilde\Omega_n}\int_Y \overline{\boldS^0}(x,y):\boldU(x,y)\dy\dx,\\
			&\lim_{k\to+\infty} IV^k=\int_{\tilde\Omega_n}\int_Y \boldS(y,\boldU(x,y)):\boldU(x,y)\dy\dx.
		\end{split}
	\end{equation*}
	Thus one obtains for any $n\in\eN$ and $\boldU\in L^p(\tilde\Omega_n;C_{per}(Y;\Rdsym))$
		\begin{equation}\label{MonotIneq}
		\int_{\tilde\Omega_n}\int_Y\left(\overline{\boldS^0}(x,y)-\boldS(y,\boldU(x,y))\right):\left(\boldD\boldu(x)+\boldD_y\boldu^0(x,y)-\boldU(x,y)\right)\dy\dx\geq 0.
	\end{equation}
To be able to apply Minty's trick, we need \eqref{MonotIneq} to be satisfied for any $\boldU\in L^p(\tilde\Omega_n\times Y;\Rdsym)$. In order to obtain that we consider $\boldU\in L^p(\tilde\Omega_n\times Y;\Rdsym)$ and $\{\boldU^k\}\subset L^p(\tilde\Omega_n;C_{per}(Y;\Rdsym))$ such that $\boldU^k\rightarrow \boldU$ in $L^p(\tilde\Omega_n\times Y;\Rdsym)$. Then we have due to the growth of $\boldS$ and theory of Nemytskii operators that $\boldS(y,\boldU^k)\rightarrow \boldS(y,\boldU)$ in $L^{p'}(\tilde\Omega_n\times Y;\Rdsym)$. Therefore one deduces the accomplishment of \eqref{MonotIneq} for any $\boldU\in L^p(\tilde\Omega_n\times Y;\Rdsym)$. Minty's trick yields that $\overline{\boldS^0}(x,y)=\boldS(y,\boldD\boldu(x)+\boldD_y\boldu^0(x,y))$ for almost all $(x,y)\in\tilde\Omega_n\times Y$. Since $|\Omega\setminus\Omega_n|\rightarrow 0$, $\{\Omega\setminus\Omega_n\}_{n=1}^{+\infty}$ is a decreasing sequence of measurable sets and $O\Subset\Omega$ was arbitrary, we have for almost all $(x,y)\in\Omega\times Y$ $\overline{\boldS^0}(x,y)=\boldS(y,\boldD\boldu(x)+\boldD_y\boldu^0(x,y))$. 
\end{proof}
Let us note that we have simultaneously proven the following lemma concerning the existence of a weak solution of the problem \eqref{ProbHom}.
\begin{Lemma}
		Let $\Omega\subset\Rd$ be a bounded Lipschitz domain, $\boldF\in L^{p'}(\Omega;\Rdsym)$ and $p>\frac{2d}{d+2}$, Assumption~\ref{Ass:S} be fulfilled and $s$ be determined by
\eqref{Def:s}. Then there exists a weak solution $(\boldu,\pi)$ of the problem \eqref{ProbHom}, which is a pair $(\boldu,\pi)\in W^{1,p}_{0,\div}(\Omega;\Rd)\times L^s(\Omega)$ such that
for any $\boldw\in C^\infty_0(\Omega;\Rd)$
	\begin{equation*}
		\int_\Omega (\hat{\boldS}(\boldD\boldu)-\boldu\otimes\boldu-\pi \boldI):\boldD\boldw = \int_\Omega \boldF:\boldD\boldw.
	\end{equation*}
\end{Lemma}

\section*{Acknowledgements}
M.~Bul\'{\i}\v{c}ek thanks the project GA\v{C}R16-03230S financed 
by Czech Science Foundation. M.~Bul\'{\i}\v{c}ek and P. Kaplick\'y are members of the Ne\v{c}as Center for
Mathematical Modeling. M. Kalousek was supported by the grant SVV-2016-
260335 and the project UNCE 204014.

\def\cprime{$'$} \def\ocirc#1{\ifmmode\setbox0=\hbox{$#1$}\dimen0=\ht0
  \advance\dimen0 by1pt\rlap{\hbox to\wd0{\hss\raise\dimen0
  \hbox{\hskip.2em$\scriptscriptstyle\circ$}\hss}}#1\else {\accent"17 #1}\fi}
  \def\ocirc#1{\ifmmode\setbox0=\hbox{$#1$}\dimen0=\ht0 \advance\dimen0
  by1pt\rlap{\hbox to\wd0{\hss\raise\dimen0
  \hbox{\hskip.2em$\scriptscriptstyle\circ$}\hss}}#1\else {\accent"17 #1}\fi}
  \def\ocirc#1{\ifmmode\setbox0=\hbox{$#1$}\dimen0=\ht0 \advance\dimen0
  by1pt\rlap{\hbox to\wd0{\hss\raise\dimen0
  \hbox{\hskip.2em$\scriptscriptstyle\circ$}\hss}}#1\else {\accent"17 #1}\fi}
\providecommand{\bysame}{\leavevmode\hbox to3em{\hrulefill}\thinspace}
\providecommand{\MR}{\relax\ifhmode\unskip\space\fi MR }
\providecommand{\MRhref}[2]{%
  \href{http://www.ams.org/mathscinet-getitem?mr=#1}{#2}
}
\providecommand{\href}[2]{#2}

\end{document}

%% file: math-e.tex
\renewcommand{\div}{\operatorname{div}}

\newcommand{\curl}{\operatorname{curl}}

\newcommand{\bD}{\mathbf{D}}

%% file: zhikov_CE_arx.bbl
\begin{thebibliography}{10}

\bibitem{AcFu}
E.~Acerbi and N.~Fusco, \emph{An approximation lemma for {$W\sp {1,p}$}
  functions}, Material instabilities in continuum mechanics ({E}dinburgh,
  1985--1986), Oxford Sci. Publ., Oxford Univ. Press, New York, 1988, pp.~1--5.
  \MR{970512 (89m:46060)}

\bibitem{Allaire}
G.~Allaire, \emph{Homogenization and two-scale convergence}, SIAM J. Math.
  Anal. \textbf{23} (1992), no.~6, 1482--1518. \MR{1185639 (93k:35022)}

\bibitem{BM}
J.~M. Ball and F.~Murat, \emph{Remarks on {C}hacon's biting lemma}, Proc. Amer.
  Math. Soc. \textbf{107} (1989), no.~3, 655--663. \MR{984807 (90g:46064)}

\bibitem{BGMS}
M.~Bul{\'{\i}}{\v{c}}ek, P.~Gwiazda, Josef M{\'a}lek, and
  A.~{\'S}wierczewska-Gwiazda, \emph{On unsteady flows of implicitly
  constituted incompressible fluids}, SIAM J. Math. Anal. \textbf{44} (2012),
  no.~4, 2756--2801. \MR{3023393}

\bibitem{DHHR}
L.~Diening, P.~Harjulehto, P.~H{\"a}st{\"o}, and
  M.~R{\ocirc{u}}{\v{z}}i{\v{c}}ka, \emph{Lebesgue and {S}obolev spaces with
  variable exponents}, Lecture Notes in Mathematics, vol. 2017, Springer,
  Heidelberg, 2011. \MR{2790542}

\bibitem{DMS}
L.~Diening, J.~M{\'a}lek, and M.~Steinhauer, \emph{On {L}ipschitz truncations
  of {S}obolev functions (with variable exponent) and their selected
  applications}, ESAIM Control Optim. Calc. Var. \textbf{14} (2008), no.~2,
  211--232. \MR{2394508 (2009e:35054)}

\bibitem{DRW}
L.~Diening, M.~R{\ocirc{u}}{\v{z}}i{\v{c}}ka, and J.~Wolf, \emph{Existence of
  weak solutions for unsteady motions of generalized {N}ewtonian fluids}, Ann.
  Sc. Norm. Super. Pisa Cl. Sci. (5) \textbf{9} (2010), no.~1, 1--46.
  \MR{2668872}

\bibitem{DieUh}
J.~Diestel and J.J. Uhl, \emph{Vector measures}, Mathematical surveys and
  monographs, American Mathematical Society, 1977.

\bibitem{FeiNov}
E.~Feireisl and A.~Novotn{\'y}, \emph{Singular limits in thermodynamics of
  viscous fluids}, Advances in Mathematical Fluid Mechanics, Birkh\"auser
  Verlag, Basel, 2009. \MR{2499296 (2011b:35001)}

\bibitem{FMS}
J.~Frehse, J.~M{\'a}lek, and M.~Steinhauer, \emph{On analysis of steady flows
  of fluids with shear-dependent viscosity based on the {L}ipschitz truncation
  method}, SIAM J. Math. Anal. \textbf{34} (2003), no.~5, 1064--1083
  (electronic). \MR{2001659}

\bibitem{Murat}
F.~Murat, \emph{Compacit\'e par compensation}, Ann. Scuola Norm. Sup. Pisa Cl.
  Sci. (4) \textbf{5} (1978), no.~3, 489--507. \MR{506997}

\bibitem{Nguetseng}
G.~Nguetseng, \emph{A general convergence result for a functional related to
  the theory of homogenization}, SIAM J. Math. Anal. \textbf{20} (1989), no.~3,
  608--623. \MR{990867 (90j:35030)}

\bibitem{RuzK}
M.~R{\ocirc{u}}{\v{z}}i{\v{c}}ka, \emph{Electrorheological fluids: modeling and
  mathematical theory}, Lecture Notes in Mathematics, vol. 1748,
  Springer-Verlag, Berlin, 2000. \MR{1810360 (2002a:76004)}

\bibitem{Stein2}
E.~M. Stein, \emph{Singular integrals and differentiability properties of
  functions}, Princeton Mathematical Series, No. 30, Princeton University
  Press, Princeton, N.J., 1970. \MR{0290095 (44 \#7280)}

\bibitem{Tartar}
L.~Tartar, \emph{Compensated compactness and applications to partial
  differential equations}, Nonlinear analysis and mechanics: {H}eriot-{W}att
  {S}ymposium, {V}ol. {IV}, Res. Notes in Math., vol.~39, Pitman, Boston,
  Mass.-London, 1979, pp.~136--212. \MR{584398}

\bibitem{Visintin}
A.~Visintin, \emph{Towards a two-scale calculus}, ESAIM Control Optim. Calc.
  Var. \textbf{12} (2006), no.~3, 371--397 (electronic). \MR{2224819
  (2007b:35034)}

\bibitem{Zhikov}
V.~V. Zhikov, \emph{Homogenization of a {N}avier-{S}tokes-type system for
  electrorheological fluid}, Complex Var. Elliptic Equ. \textbf{56} (2011),
  no.~7-9, 545--558. \MR{2832202 (2012i:35022)}

\end{thebibliography}
